    	\documentclass[11pt]{amsart}

\usepackage{amsrefs}

     \usepackage[all]{xy}

     \usepackage{enumerate}

     \newenvironment{problist}
{\begin{enumerate}[(a)]}
{\end{enumerate}}

\usepackage{eucal}

     	\usepackage{amssymb}
     \usepackage{amsmath}
    	\usepackage{amsfonts}
	\usepackage{latexsym}   		
	
    \newtheorem{thm}{Theorem}

    \newtheorem{lem}[thm]   {Lemma}
    \newtheorem{cor}[thm]   {Corollary}

    \newtheorem{prop}[thm]  {Proposition}

    \newcommand{\term}[1]   {{\bf #1}\index{#1}}

    \newcommand{\inclds}     {\hookrightarrow}

                \newcommand{\K}        {\mathcal{K}}
                \newcommand{\LL}      {\mathcal{L}}

\newcommand{\R}  {\mathcal{R}}

\newcommand{\T}  {\mathcal{T}}

\newcommand{\J}{\mathcal{J}}

    \newcommand{\ZZ}    {\mathbb{Z}}

    \newcommand{\NN}    {\mathbb{N}}
    
    \newcommand{\QQ}    {\mathbb{Q}}

    \newcommand{\s}     {\Sigma}
    
    \newcommand{\smsh}  {\wedge}
    \newcommand{\om}    {\Omega}
    \newcommand{\cross} {\times}
    \newcommand{\wdg}   {\vee}

    \newcommand{\of}    {\circ}
    \newcommand{\id}    {\mathrm{id}}

    \newcommand{\twdl}  {\widetilde}

    \newcommand{\sseq}  {\subseteq}

\renewcommand{\P}{\mathcal{P}}

\DeclareMathOperator{\hocolim}{hocolim}

 \DeclareMathOperator{\im}    {im}

    \DeclareMathOperator{\Tor}{Tor}

        \DeclareMathOperator{\conn}{conn}

        \DeclareMathOperator{\map}    {map}

\newcommand{\mmathand}{\qquad\mbox{and}\qquad}

\begin{document}
	
	


\title{Idempotent functors that preserve cofiber sequences
and split suspensions
}



\author{Jeffrey Strom}
\address{Department of Mathematics\\
Western Michigan University\\
Kalamazoo, MI\\
49008-5200 USA}
\email{Jeff.Strom@wmich.edu}	




\begin{abstract}
We   show that  an $f$-localization functor
 $L_f$ commutes with cofiber sequences
of $(N-1)$-connected finite complexes if and only if its restriction
to the collection of $(N-1)$-connected 
finite complexes is $R$-localization for some unital subring
$R\sseq\QQ$.  
This leads to    a homotopy-theoretical  characterization
of the rationalization functor:  the restriction of 
$L_f$ to  simply-connected 
spaces (not just the finite complexes)
 is rationalization if and only if  $L_f(S^2)$ is nontrivial
and simply-connected, $L_f$ preserves cofiber sequences
 of simply-connected finite complexes, and 
 for each simply-connected finite complex
$K$, $\s^k L_f(K)$ splits as a wedge of copies
of $L_f(S^n)$ for large enough $k$
and various values of $n$.
\end{abstract}

\maketitle
\pagestyle{myheadings}
\markboth{{J. Strom}}
{{Idempotent functors that preserve cofiber sequences
and split suspensions}}

%
%

\section*{Introduction}

Let $f:P\to Q$ be a map from one CW complex to another.
The $f$-localization functor $L_f$ 
is the universal example of a homotopically idempotent
functor from spaces to spaces
which carries $f$ to a weak equivalence.  
These functors are the primary---conceivably the 
only\footnote{It has been shown 
\cite{MR2166179} that it is impossible 
to prove in ZFC that all homotopically idempotent functors
take the form $L_f$ for some $f$; but  in the presence of
Vop\v enka's Principle (which is thought, but not known, to 
be consistent with ZFC),   every homotopically idempotent
functor \textit{is} of the form
 $L_f$ for some well-chosen map $f$.}---examples of homotopically 
idempotent functors.
 
The $R$-localization functors
are among  
the earliest and 
best-behaved examples of localization functors, 
defined for (unital)
subrings $R\sseq \QQ$;
we denote them by $L_R$.
There are various constructions of these,   but all of them agree 
on simply-connected spaces.  They were constructed 
with the intention of lifting to spaces the algebraic operation 
of $R$-localization applied to homotopy and homology groups.
This prescribed algebraic effect guarantees two 
nice homotopy-theoretical
properties of $R$-localization:    $L_R$ 
carries cofiber sequences of simply-connected spaces to 
cofiber sequences and similarly for fiber sequences.

Rationalization is the special case $R = \QQ$, and here the theory
is most powerful.  The central theorems of 
Quillen \cite{MR0258031}
and Sullivan
\cite{MR0442930} show that the homotopy theory  of
simply-connected rational spaces
(i.e., 
simply-connected spaces for which $X\to L_\QQ(X)$
is a weak equivalence) is perfectly modeled   
by the algebraic homotopy theories
of DGLAs  and of CDGAs, respectively.    Rational spaces have 
two additional   
homotopy-theoretic properties beyond 
   those enjoyed by all simply-connected $R$-local spaces:  
 their suspensions
 split as  wedges of 
copies of $L_\QQ(S^n)$ for various values of $n$; and,
 dually, their loop spaces  split as  
products of copies of $K(\QQ,n)$.

Many
 noncomputational theorems about  
 rationalizations which were
  first proved via the algebraic machinery provided
by Quillen and Sullivan can also be proved
by 
appealing to the basic homotopy-theoretical properties 
noted above.  For example,  the 
celebrated Mapping Theorem for Lusternik-Schnirelmann
 category   was 
originally proved using Sullivan 
models \cite{MR664027}, but later 
a much simpler proof was found that was based on the splitting
of loop spaces \cite{MR790674}.
Examples like this led the author to wonder    if 
any localization 
  functor $L_f$ that 
  satisfies the four homotopy-theoretical properties
  listed above must be related to rationalization in some way.
   
 In fact,  rationalization is
   determined---o\-ver\-de\-ter\-min\-ed!---by these properties:
  it is the only
  localization functor of simply-connected spaces that 
  preserves   cofiber sequences and   splits suspensions.  
  Furthermore,  the restriction 
  of $R$-localization to simply-connected finite complexes
  is characterized by the fact that it commutes with cofiber sequences.
  These are the main results of this paper.

We write $\K(N)$ for  the collection of all $(N-1)$-connected
finite complexes.


 \begin{thm}
 \label{thm:Main1}
  Let  $f:P\to Q$ be a map of CW complexes.
Let $N\geq 2$ and suppose $L_f(S^N)$ is 
simply-connected
 and not contractible.
Then 
 the following are equivalent:
 \begin{enumerate}
 \item
the restriction of 
$L_f$ to $\K(N)$
 is $R$-localization for
some (unital) subring $R\sseq \QQ$,
 \item
 $L_f$ commutes with cofiber sequences\footnote{For any 
 $f$, there is a  natural comparison   map  $\xi:C_{L_f(\alpha)} 
 \to L_f(C_\alpha)$;  $L_f$ commutes with the cofiber sequence 
 $X\xrightarrow{\alpha} Y \to C_\alpha$ if this transformation
 is a weak equivalence.}
 of
$(N-1)$-connected finite complexes.
  \end{enumerate}
 \end{thm}
 
 

%

Taking $N=2$ in
Theorem \ref{thm:Main1}, we obtain 
 a   characterization of the 
rationalization of simply-connected spaces
in terms of the elementary notions of 
homotopy theory.

\begin{thm}
\label{thm:Qchar}
The restriction of a localization functor $L_f$ 
to simply-connected
spaces is rationalization if and only if 
the following three conditions hold:
\begin{itemize}
\item
$L_f(S^2)$ is  simply-connected and not weakly contractible,
\item
$L_f$ commutes with cofiber sequences  of 
simply-connected finite
complexes, and
\item
if $K$ is a simply-connected finite complex, then for some 
$k\in \NN$, the suspension $\s^k L_f(K)$
splits as a wedge of copies of $L_f(S^n)$ for various values of 
$n$.
\end{itemize}
 \end{thm}

  We conclude the introduction with two bits of speculation.
  
Our proof of Theorem \ref{thm:Main1} comes
 very close to showing that $L_f$ commutes
  with cofiber sequences (and hence restricts to 
  $R$-localization on $\K(N)$) 
  if and only if  $L_f$
     `respects the smash and suspension 
  structure   of spheres'---that is, 
  if $\s L_f(S^{n}) $ and $L_f(S^n)\smsh L_f(S^m)$
  are $f$-local for all $m,n\geq N$.
  Using a relative version of the  theory of resolving classes\footnote{See
  \cite{CLARK} for an account of the absolute theory.}, 
  we have been able to prove this equivalence 
  under the additional assumption that
   $f$
  factors up to homotopy through a finite-dimensional complex
  (we have not included that proof here). 
Is 
it true  for all $f$?

Second, let us say that 
$L_f$ `respects connectivity'  if for any $n$-connected
 $K\in K(N)$, $L_f(K)$ is also $n$-connected.
 If $L_f$ respects connectivity, must it 
satisfy the conditions of Theorem \ref{thm:Main1}?

\medskip
 
 \noindent{\textbf{Acknowledgement.}}
 The author gratefully acknowledges the 
 assistance of Steven Landsburg (via MathOverflow)
 on the algebraic implications of the vanishing of $\Tor$.
Javier 
 Guti\'errez provided   valuable feedback 
 on an earlier version of this paper, as did the 
referee.

\section{Preliminaries}

We'll work  in a fixed convenient category $\T_*$ of pointed 
topological spaces; for example $\T_*$ could be the
category of compactly generated weak Hausdorff
spaces.
If so inclined, the reader  
 may pretend that this paper
was written simplicially.

We write $\conn(X) = n$ if $X$ is $n$-connected but not 
$(n+1)$-connected.

 \subsection{Localization Functors}
Let $f:P\to Q$ be a map between pointed CW complexes.  
A pointed space $X$ is said to be \term{$f$-local} if 
the induced map 
\[
f^* : \map_*(Q,X) \longrightarrow\map_*(P,X)
\]
of pointed mapping spaces
is a weak homotopy equivalence;  a map
$q:X\to Y$ is called an \term{$f$-equivalence}
if for every 
$f$-local space $Z$
the induced map
\[
q^* : \map_*(Y,Z) \longrightarrow\map_*(X,Z)
\]
is a weak homotopy equivalence.  A map $i:X\to L$ is 
said to be an
\term{$f$-localization} of $X$ if $L$ is $f$-local and $i$
is an $f$-equivalence.

The following important properties follow easily  
from the definitions.  

\begin{lem}
\label{lem:factorization}
{\quad}
\begin{problist}
\item
An $f$-equivalence between $f$-local spaces
is a weak equivalence.
\item
Let  $i:X\to L$ be $f$-localization 
and let  $g:X\to Z$ with  $Z$ $f$-local.
Then in the diagram
\[
\xymatrix{
X\ar[rr]^-g \ar[d]_i && Z
\\
L\ar@{..>}@/_/[rru]_\gamma
}
\]
there is a  map $\gamma$, unique up to homotopy,
making the triangle commute up to homotopy.
\end{problist}
\end{lem}

%
%
%



An \term{coaugmented functor} is a functor 
$F:\T_*\to \T_*$ equipped with a natural transformation
(its  coaugmentation)
$\iota :\id \to F$.  A coaugmented functor $F$ is 
\term{homotopically idempotent} if for every 
$X$, the maps
\[
\iota_{F(X)}, F(\iota_X): F(X) \longrightarrow F( F(X))
\]
are homotopic to one another, and both are weak equivalences.
An \term{$f$-localization functor} is a
homotopically idempotent functor $L_f$ 
such that, for every space $X$, the coaugmentation
$\iota_X: X\to L_f(X)$ is an $f$-localization of $X$.

The main existence  theorem is as follows
\cite[Thm.1.A.3]{MR1392221}.

\begin{thm}[Bousfield, Farjoun]
\label{thm:LfExists}
For any map $f:P\to Q$    between CW complexes,
there exists an $f$-localization functor $L_f$.
\end{thm}
 
Here are some basic properties of $f$-localization.

 \begin{prop}
 \label{prop:Lfproperties}
Let $f:P\to Q$ be a map of CW complexes.
 \begin{problist}
 \item
 If $q: X\to Y$ is an $f$-equivalence, then so is 
 $q\smsh \id_Z$ for any space $Z$.
 \item
 $L_f(q)$ is a weak equivalence if and only if $q$
 is an $f$-equivalence.
 \item
 For any $X$ and $Y$, the natural   map
 $L_f(X\cross Y) \to L_f(X) \cross L_f(Y)$ is a weak equivalence.
 \end{problist}
 \end{prop}

 \begin{proof}
 Parts (a) and 
 (b)
 are Example 1.D.5 and Proposition 1.C.5
 in \cite{MR1392221}, respectively.
Part 
(c)
is \cite[1.A.8(e.4)]{MR1392221}.
\end{proof}

If $q:X\to Y$ is an $f$-equivalence, then 
 every $f$-local space is also $q$-local, and 
 so
 the transformation $j_{L_f}$ in the square
 \[
\xymatrix{
	\id	%
                 \ar[d]_{ \iota }  
                 \ar[rr]^-{  j  }
		\ar@{}[rrd]|{     }
&& 
	L_q	%
                  \ar[d]^{  L_f( \iota ) }
\\
L_f		%
	         \ar[rr]^-{j_{L_f}  }
&&  
L_q\of L_f		%
}
\]
evaluates to a  weak equivalence 
  $j_{L_f(X)}:L_f(X) \to L_q(L_f(X))$  
for every 
space $X$.  
Thus we say that
$\iota:\id\to L_f$   
factors
through
$j:\id\to L_q$
   `up to weak homotopy equivalence'.
Since $L_q\of L_f(X) \sim L_f(X)$ for all $X$, 
the functor $L_q\of L_f$ is a perfectly good choice
of $f$-localization functor.  Thus we will abuse 
notation by silently redefining $L_f$
and referring to $L_f(\iota)$ as a \term{comparison map} 
  $\xi: L_q\to L_f$.
%


\begin{prop}
\label{prop:flocalisqlocal}
Suppose $q$ is an $f$-equivalence.
If $L_q(X)$ is $f$-local, then
the comparison $\xi: L_q(X) \to L_f(X)$ is a weak equivalence.
\end{prop}

\begin{proof}
Any $q$-equivalence   is \textit{a fortiori} an $f$-equivalence.
Since $X\to L_q(X)$ is a $q$-equivalence by definition
and $L_q(X)$ is $f$-local by hypothesis, 
 we have  weak equivalences
$L_f(X) \to L_f(L_q(X))\gets L_q(X)$.
\end{proof}

We conclude with two useful results about 
the behavior  of  $f$-localization with respect to 
  homotopy  colimits and homotopy limits.
 
  \begin{prop}
 \label{prop:cutecofibertrick}
 Let $f:P\to Q$ be a map of CW complexes, and write
 $\LL$
 for the full subcategory of $\T_*$ whose objects
 are the $f$-local spaces.
 Let   $F: \J\to \LL$ be a (small) diagram.  
  If $\hocolim L_f\of F\in \LL$, then 
  the comparison map
  $\hocolim F\to \hocolim L_f\of F$
  induced by the coaugmentations
  is also an  $f$-localization.
  \end{prop}
  
%
  
Proposition \ref{prop:cutecofibertrick} 
 is proved in \cite[Thm. 1.D.3]{MR1392221}.
It implies that if 
$A\xrightarrow{\alpha} B\to C$ is a cofiber sequence and 
$C_{L_f(\alpha)}$ is $f$-local, then the natural
comparison map
$C_\alpha  \to C_{L_f(\alpha)}$ is $f$-localization.

 The proof  of the following result is entirely straightforward.

\begin{prop}
\label{prop:LResolving}
For any $f:P\to Q$
the   subcategory
$\LL = \im(L_f)$  is closed under weak homotopy equivalences and 
homotopy limits.
\end{prop}

In the language of \cite{CLARK}, Proposition
\ref{prop:LResolving} says that $\LL$ is a \textit{resolving class}.

\subsection{$R$-Localization}\label{subsection:Rloc}
Localization with respect to a ring $R\sseq\QQ$
is   
a special case of localization of considerable importance.
We record its definition and basic properties here.

We write $p_n:S^n\to S^n$ for the degree $p$ map 
from the $n$-sphere to itself ($p$ may
be any integer, but we usually take $p$ to be a prime).
Now for a subring $R\sseq \QQ$, let 
$\P(R)$ denote the set of all primes that are invertible in $R$, 
and   define \term{$R$-localization} to be  $L_{q_R}$
where 
\[
q_R = 
%
 \bigvee_{p \in \P(R)} 
p_1
 : 
\bigvee_{p \in \P(R)} 
S^1
\longrightarrow
\bigvee_{p\in \P(R)} 
 S^1.
\]
There are other (different) definitions for rationalization, but they
all agree up to weak homotopy equivalence on simply-connected
spaces.  We write $L_R$ for 
 $R$-localization
and call a $q_R$-equivalence
an \term{$R$-equivalence}.

It is easy to detect simply-connected
$R$-local spaces in terms of their standard
algebraic invariants.
An abelian group $G$ is said to be \term{$R$-local} if 
the map 
$G\otimes_\ZZ \ZZ \to G\otimes_\ZZ R$ 
induced by the inclusion $\ZZ\inclds R$ is an
isomorphism.

\begin{prop}[Sullivan]
\label{prop:algchar}
Let $R\sseq \QQ$ be a subring.  
If $X$ is simply-connected, then the following are equivalent:
\begin{enumerate}
\item
$X$ is $R$-local,
\item
$\pi_n(X)$ is $R$-local for  all $n\geq 2$,
\item
$\twdl H_n(X;\ZZ)$ is $R$-local for all  $n\geq 2$.
\end{enumerate}
\end{prop}

Together with the exactness of the $R$-localization 
of abelian groups, Proposition \ref{prop:algchar}
implies the two basic homotopy-theoretical 
properties of $R$-lo\-cal\-i\-za\-tion of spaces.

\begin{thm}[Sullivan]
\label{thm:Qprops}
If $R\sseq \QQ$, then 
$L_R$ commutes with both fiber sequences and 
cofiber sequences of simply-connected 
spaces.
\end{thm}

These results are well-known;  they may be found in 
  \cite[Thm. 2.1]{MR0442930} and \cite[Thm. 2.5]{MR0442930}.

\begin{cor}
\label{cor:Rlocalviaq}
If $X$ is $(N-1)$-connected, then 
  $X\to L_{\s^{N-1} q_R}(X)$   is  $R$-localization.
     \end{cor}

In the special case $R = \QQ$ it is well-known that 
we can say 
quite a lot more.

\begin{thm}
\label{thm:Qprops}
For any simply-connected space $X$, 
\begin{problist}
\item
$\s X_\QQ\sim \bigvee_\alpha S^{n_\alpha}_\QQ$
for various values of $n_\alpha$, and
\item
$\om X_\QQ\sim
\prod_\alpha K(\QQ,m_\beta)$ for various values of $m_\beta$.
\end{problist}
\end{thm}

This implies a closure property for resolving classes
of rational spaces.

 \begin{lem}
 \label{lem:Qresclass}
 Let $\R$ be a resolving class.\footnote{That is, $\R$ is nonempty
 and closed under weak equivalence and homotopy limits;
 according to Proposition \ref{prop:LResolving}, 
 $\LL= \im(L_f)$ is such a class.} 
 If there is a simply-connected
 rational space $X\in \R$ with $\pi_n(X) \neq 0$, 
 then $\R$ contains all 
 simply-connected rational spaces with $\pi_k(Y) = 0$
 for $k\geq n$.
 If for each $n\in \NN$ there is a rational space 
 $X_n\in\R$   with $\pi_k(X_n)\neq 0$ for some $k\geq n$, 
 then $\R$ contains all simply-connected rational spaces.
 \end{lem}
 
 \begin{proof}
Since $K(\QQ,n-1)$ is a retract of $\om X\in \R$,   
   $K(V, k)\in \R$ for all rational vector spaces 
 $V$ and all $k< n$.
 Now the result follows easily by induction on the Postnikov
 system of $Y$.
 \end{proof}

\section{Proof of Theorem \ref{thm:Main1}}

 It is well-known that (1) implies (2).
 
Suppose we have (2).
We begin our proof of (1)
by  studying the $f$-localization of $(N-1)$-connected
spheres.
Applying (2) to the cofiber sequences
\[
S^{n} \to * \to S^{n+1} \mmathand S^{n} \xrightarrow{*}
  S^m \to S^{m}\wdg S^{n+1},
\]
  reveals that $L_f(S^{n+1}) \simeq\s L_f(S^{n})$
and $L_f(S^m\wdg S^{n+1}) \simeq L_f(S^m)\wdg L_f(S^{n+1})$, as long 
as $m,n\geq N$.    
Inductively, we obtain weak equivalences
$S^m \smsh L_f(S^n)
\xrightarrow{\sim}  L_f(S^m\smsh S^n)$ 
for all $m\geq 1$ and $n\geq N$.
Next (still with  $m,n\geq N$)
apply $L_f$ to the cofiber sequence
\[
S^m\wdg S^{n+1} \to S^m\cross S^{n+1} \to S^m \smsh S^{n+1}
\]
to obtain
\begin{eqnarray*}
L_f (S^m \smsh S^{n+1}) &\sim&
L_f(S^m\cross S^{n+1}) / L_f(S^m\wdg S^{n+1})
\\
&\sim &
(L_f(S^m)\cross L_f(S^{n+1})) / (L_f(S^m)\wdg L_f(S^{n+1}))
\\
&\sim &
L_f(S^m)\smsh L_f(S^{n+1}), 
\end{eqnarray*}
using Proposition \ref{prop:Lfproperties}(c).

We use these equivalences to 
 construct the solid arrow part of the diagram
 \[
 \xymatrix{
	S^m \smsh S^{n+1}  	%
                 \ar[d]_{   \iota_{S^m\smsh S^{n+1}} }  
                 \ar[rr]^-{   \id_{S^m} \smsh \iota_{S^{n+1}} }
&& 
		S^m \smsh L_f(S^{n+1})	
		\ar[lld]_{    \iota_{S^m\smsh L_f(S^{n+1})}  }%
                  \ar[d]^{    \iota_{S^m} \smsh \id_{L_f(S^{n+1})} }
\\
		L_f(S^m\smsh S^{n+1})%
&&  
	L_f(S^m )\smsh L_f(S^{n+1})	%
		         \ar@{..>}[ll]_-{\epsilon_{m,{n+1}}  }
}
\]
for $m,n \geq N$.
Proposition \ref{prop:Lfproperties}(a) implies that
  the solid arrows in the diagram 
 are $f$-equivalences.
Since
 $L_f(S^m)\smsh L_f(S^{n+1})\in \LL$, 
the dotted arrow 
$\epsilon_{m,n+1}$ exists and is unique up to homotopy 
by the universal factorization property of $L_f$
given in Lemma \ref{lem:factorization}(b).
Further, the commutativity of the diagram implies that 
that $\epsilon_{m,n+1}$ is also an $f$-equivalence.  
Since all the spaces   in the  lower
right triangle are $f$-local,  all three maps
in that triangle  are   weak equivalences
by Lemma \ref{lem:factorization}(a).

We next determine the connectivity of $L_f(S^n)$.
 Write $\conn (L_f(S^N)) = c$;  since $L_f(S^N)$
  is simply-connected, 
 we have 
 \[
 \conn( L_f(S^n)) = \conn ( \s^{n-N} L_f(S^N) )  = c + (n-N)
 \]
 for all $n \geq N$, so we may work with $n > N$, 
and we have the equation
\[
\conn ( L_f(S^n) \smsh L_f(S^n)) = 2(c + (n-N)) +1.
\]
Since $L_f(S^n)$ is simply-connected, 
the homotopy-commutative diagram
of weak equivalences
\[
 \xymatrix{
    S^n\smsh  L_f(S^n)    %
        \ar[r]^-{ \sim   }
        \ar@/_/[rd]_{  \sim }
        &
   L_f(S^n) \smsh L_f(S^n)     %
        \ar[d]_{  \sim  }^{\epsilon_{n,n}}
        &
        L_f(S^n) \smsh S^n \ar@/^/[ld]^\sim\ar[l]_-\sim
\\
		&
	L_f(S^{2n})	 %
}
\]
shows that  
\[
n + (c +(n - N)) 
= \conn( \s^n L_f(S^n)) 
= \conn( L_f(S^n)\smsh L_f(S^n) ) 
= 2(c + (n-N)) +1, 
\]
so  $c = N-1$ and $\conn(L_f(S^n)) = n-1$ for all $n\geq N$.

  Now let $n > N$ be an even integer and write
  $L = L_f(S^n)$.
  We determine the graded abelian group  
  $A_* = \s^{-n} \twdl H_*(L ;\ZZ)$.
  Applying integral homology to the     diagram above
 results in the diagram
\[
 \xymatrix{
    \ZZ \otimes_\ZZ A_*     %
        \ar[r]^-{     }
        \ar[d]_\cong^\kappa
        &
    A_*\otimes_\ZZ A_*    %
        \ar[d]^{  \kappa   }
        & 
        A_*\otimes_\ZZ \ZZ
        \ar[l] 
        \ar[d]^\kappa_\cong
\\
 A_*\ar[r]^-\cong\ar@/_/[rd]_\cong
& 
\s^{-2n} \twdl H_*(L\smsh L;\ZZ)\ar[d]^\cong
&
 A_*\ar[l]_-\cong\ar@/^/[ld]^\cong
\\
		&
	A_* 	.%
}
\]
The commutativity of the diagram   implies 
 that the exterior product map 
 $\kappa: A_*\otimes A_* \to \s^{-2n}
 \twdl H_*(L\smsh L;\ZZ)$
is an isomorphism ($\kappa$ is injective by the K\"unneth theorem).
Write $A_+$ for the positive degree part of $A_*$.  
Since $(\ZZ \otimes_\ZZ A_+)\cap (A_+\otimes_\ZZ \ZZ) = 0$ 
in $A_*\otimes_\ZZ A_*$, 
it must be that $A_+\cong \ZZ \otimes_\ZZ A_+  =  0$ 
or else the vertical composite could not be injective in positive 
degrees.
Thus
$A_*$ is some group $R$ concentrated in 
 degree zero, and the  vertical
composite  
reduces to an isomorphism
$\mu: R\otimes_\ZZ R\xrightarrow{\cong} R$.

It follows directly from the commutativity and 
associativity of the smash product that 
the map $\mu$ gives $R$ the structure of a 
commutative ring.  Since this product is an isomorphism, 
$R$ is a \textit{solid ring} (defined and studied
by Bousfield and Kan in \cite{MR0308107}).
Applying the K\"unneth theorem to compute 
$A_1$ (which we 
   know to be zero), we find that $\Tor_\ZZ(R,R) = 0$;
we
show that if $R$ is a solid ring with $\Tor_\ZZ(R,R) = 0$, 
then $R$ is isomorphic to a subring of $\QQ$.
Let $T\sseq R$ be the torsion subgroup.
We cannot have  $T = R$,  for then $R\cong {\mathbb Z}/d$ by
 \cite[Lem. 3.6]{MR0308107}, and this  forces
  $\Tor_\ZZ(R,R) \neq 0$.  
Therefore we may apply \cite[Lem. 3.10]{MR0308107},   
to  the short exact sequence
\[
0\rightarrow T\longrightarrow R \longrightarrow R/T\rightarrow 0
\]
to discover that 
$R/T$ is isomorphic to a subring   $S\sseq \QQ$ (and is 
therefore a  flat $\ZZ$-module) and that $T$ is a sum of cyclic groups.  
It follows that if $T\neq 0$, then  $\Tor_\ZZ(T,T) \neq 0$, and that
$R \cong S \oplus T$ as
abelian groups.
Therefore we may compute
\begin{eqnarray*}
\Tor_\ZZ ( R, R) 
&\cong&
\Tor_\ZZ(S, S) \oplus 
\Tor_\ZZ (S, T) \oplus 
\Tor_\ZZ (T,S) \oplus
\Tor_\ZZ(T,T) 
\\
&\cong&
 \Tor_\ZZ(T,T), 
\end{eqnarray*}
showing that if $\Tor_\ZZ(R,R)=0$, 
then $T = 0$ and $R \cong S\sseq \QQ$.
Thus  the $f$-localization  of $S^n$ takes the form
 $\iota_{S^n}: S^n \to M(R,n)$ for all $n\geq N$.
 
We claim   that $\iota_{S^n}$ is $R$-localization
for all $n\geq N$.
Let $\lambda:S^n \to M(R,n)$ 
represent $1\in R\cong \pi_n(M(R,n))$; then 
$\iota_{S^n} = r\cdot \lambda$ for some $r \in R$. 
Since, as we showed above,  
 $L_f$ commutes with suspension of $(N-1)$-connected
spheres,  this $r$   is the same for all $n\geq N$. 
To show that $\iota_{S^n}$ is $R$-localization, 
we need to show that
$r$ is invertible in $R$.  
Note that $r$ cannot be zero, for if  $r = 0$, 
then $\iota_{S^n}: S^n \to L_f(S^n)$ is 
nullhomotopic, forcing   $L_f(S^n)$ to  be contractible.
Since $M(R,n)$ is $f$-local, $L_f(M(R,n)) \sim M(R, n)$
and the coaugmentation $M(R,n)\to L_f(M(R,n))$
can be identified with $t \cdot \id$ for some 
invertible $t\in R$.
Writing $L_f(\lambda ) = s \cdot \id$
for some $s\in R$,
we obtain 
  the diagram
\[
\xymatrix{
		S^n %
                 \ar[d]_{ r\cdot\lambda  }  
                 \ar[rr]^-{  \lambda }
		\ar@{}[rrd]|{     }
&& 
	M(R,n)	%
                  \ar[d]^{ t \cdot \id }
\\
M(R,n)		%
	         \ar[rr]^-{ s\cdot \id}
&&  
	 M(R,n)  	. %
}
\]
This shows that $rs = t  $; and since 
$t$ is invertible, so are $r$ and $s$.
Therefore $\iota_{S^n}$ is $R$-localization
for $n\geq N$.\footnote{A similar
result for smashing localizations of spectra 
was established in \cite[Thm. 5.14]{MR2139526}.}

Now we show that $\s^{N-1} q_R$ is an $f$-equivalence; for this
it suffices to show that $p_N$ (see Section \ref{subsection:Rloc}
for notation)
is an $f$-equivalence for each 
$p\in \mathcal{P}(R)$.
Because $L_f$ commutes with cofiber sequences, 
$L_f(S^n\wdg S^n) \sim L_f(S^n)\wdg L_f(S^n)$
for $n >N$.  Thus
  the   co-H structure map $S^n \to S^n\wdg S^n$
  is carried by $L_f$ to the co-H structure map 
  $M(R,n) \to M(R, n)\wdg M(R,n)$ for $n > N$ (both structure maps
  are unique up to homotopy).
It follows that  $L_f(\s \alpha + \s \beta) = \s L_f(\alpha) + \s  L_f(\beta)$
for any $\alpha, \beta\in \pi_n(S^n)$.
Since $1_n = \id:S^n\to S^n$, we have $L_f(1_n) = \id_{M(R,n)}$, 
and therefore $L_f(p_n) = p\cdot \id_{M(R,n)}$.
Now if $p$ is invertible in $R$, then $L_f(p_n) = p \cdot \id$
is a weak equivalence $M(R,n) \to M(R,n)$ for $n > N$, and 
also for $n = N$ by the Freudenthal suspension theorem.

Since $ \s^{N-1} q_R$ is an $f$-equivalence,  there is a
comparison $\xi: L_{ \s^{N-1} q_R} \to L_f$
which evaluates to weak equivalences
$\xi_{S^n}$ for $n\geq N$.
By virtue of Proposition \ref{prop:flocalisqlocal},
to  prove  that $\xi_K$
 is a weak 
equivalence for all $K\in \K(N)$, it suffices
to show that $\K(N) \sseq \LL= \im(L_f)$.
Since $\LL$ is a resolving class, the desuspension theorem 
\cite[Thm. 8]{CLARK} shows that it suffices
to show that the class $\s \K(N)$ of suspensions of 
spaces in $\K(N)$ is contained in  $\LL$.
But $\s \K(N)\sseq \K(N+1)$, 
so it suffices to show that $\K(N+1)\sseq \LL$;
we accomplish this  by induction on the number 
of nontrivial cells in a CW decomposition of $K$. 
The initial case is the trivial case $K = *$. For the inductive step, 
fit $K$ into a cofiber sequence
 $S^n\xrightarrow{\alpha} L\to K$   
 in which $n\geq N$ and $L\in \K(N+1)$ has
strictly fewer nontrivial cells than $K$.  Then we may form the 
diagram
\[
\xymatrix{
	L_{ \s^{N-1} q_R}(S^n) %
                 \ar[d]_{\xi_{S^n} }  
                 \ar[rr]^-{ L_{ \s^{N-1} q_R}(\alpha) }
		\ar@{}[rrd]|{     }
&& 
	L_{ \s^{N-1} q_R}(L)	%
                 \ar[d]_{\xi_L  }  
                 \ar[rr]^-{  }
		\ar@{}[rrd]|{     }
&& 
C_{L_{ \s^{N-1} q_R}(\alpha)}%
                  \ar[d]^{ \zeta }
\\
	L_f(S^n)%
	         \ar[rr]^-{ L_f(\alpha) }
&& 
	L_f(L)	%
	         \ar[rr]^-{  }
&&  
	C_{L_f(\alpha)}	%
}
\]
in which the induced map $\zeta$ of 
cofibers is a weak equivalence
because $\xi_{S^n}$ and $\xi_{L}$ are weak 
equivalences by hypothesis. 
Since $L_{\s^{N-1}q_R}$ and $L_f$   commute with 
cofiber sequences of $(N-1)$-connected finite complexes, 
we have
\[
L_{ \s^{N-1} q_R}(L)\sim 
C_{L_{ \s^{N-1} q_R}(\alpha)}\sim
C_{L_f(\alpha)}\sim 
L_f(K).
\] 
This shows that 
$L_{ \s^{N-1} q_R}(K)$ is $f$-local, and we apply 
Proposition \ref{prop:flocalisqlocal}
to deduce that $\xi_K$ is a weak equivalence.
This  establishes (1) and 
completes the proof of Theorem \ref{thm:Main1}.

\section{Proof of Theorem \ref{thm:Qchar}}

 First of all, it is well-known that $\QQ$-localization
satisfies the conditions of Theorem \ref{thm:Qchar}.  
Now suppose $L_f$ satisfies those conditions.  
Then  Theorem \ref{thm:Main1} implies
that the restriction of $L_f$ to $\K(2)$
 is $R$-localization    for some $R\sseq\QQ$.  
But $R$ must be $\QQ$, for if the prime $p$  is not invertible 
in $R$, then the Moore space 
$M(\ZZ/p,n)$  is $R$-local and  cannot split as copies of
$L_f(S^{n_\alpha})  = M(R,n_\alpha)$, even  after
repeated  suspension.

The proof of Theorem \ref{thm:Main1} implies that 
$\s q_\QQ: \bigvee S^2\to \bigvee S^2$ is an $f$-equivalence.
Since the rationalizations of all simply-connected
  finite complexes are in $\LL = \im(L_f)$,
Lemma \ref{lem:Qresclass} and Corollary \ref{cor:Rlocalviaq}   
show that the rationalization of 
every simply-connected space is 
both $f$-local and $\s q_\QQ$-local.
 We deduce from Proposition \ref{prop:flocalisqlocal}
  that $X\to L_f(X)$
is $\s q_\QQ$-localization  for all simply-connected spaces  $X$;
by Corollary \ref{cor:Rlocalviaq}, this
is rationalization.



\end{document}